\pdfoutput=1
\RequirePackage{ifpdf}
\ifpdf 
\documentclass[pdftex]{sigma}
\else
\documentclass{sigma}
\fi

\numberwithin{equation}{section}

\newtheorem{Theorem}{Theorem}[section]
\newtheorem*{Theorem*}{Theorem}

\newtheorem{Lemma}[Theorem]{Lemma}
\newtheorem*{Claim}{Claim}
\newtheorem{Proposition}[Theorem]{Proposition}
\newtheorem{Question}[Theorem]{Question}

\theoremstyle{definition}

\newtheorem{Remark}[Theorem]{Remark}

\usepackage{epigraph}

\newcommand{\N}{\mathbb{N}}

\newcommand{\Z}{\mathbb{Z}}
\newcommand{\R}{\mathbb{R}}
\newcommand{\C}{\mathbb{C}}

\newcommand{\wt}{\widetilde}

\begin{document}

\allowdisplaybreaks

\newcommand{\arXivNumber}{2609.39951}

\renewcommand{\thefootnote}{}

\renewcommand{\PaperNumber}{098}

\FirstPageHeading

\ShortArticleName{Spectral Numbers in Tate Rabinowitz Floer Homology}

\ArticleName{Spectral Numbers in Tate Rabinowitz Floer Homology\footnote{This paper is a~contribution to the Special Issue on Geometry and Dynamics in memory of Will Merry. The~full collection is available at \href{https://sigma-journal.com/Merry.html}{https://sigma-journal.com/Merry.html}}}

\Author{Kai CIELIEBAK and Urs FRAUENFELDER}

\AuthorNameForHeading{K.~Cieliebak and U.~Frauenfelder}

\Address{Universit\"at Augsburg, Universit\"atsstr.~2, 86159 Augsburg, Germany}
\Email{\mail{kai.cieliebak@math.uni-augsburg.de}, \mail{urs.frauenfelder@math.uni-augsburg.de}}

\ArticleDates{Received February 15, 2026, in final form September 30, 2026; Published online October 14, 2026}

\Abstract{We show that circle-valued spectral values in Tate Rabinowitz Floer homology give rise to quantized energy levels, which recover the quantum mechanical spectrum of a~system of uncoupled harmonic oscillators.}


\Classification{53D40; 81Q20}

\renewcommand{\thefootnote}{\arabic{footnote}}
\setcounter{footnote}{0}

\begin{flushright}
{\it To the memory of our friend and colleague Will Merry}
\end{flushright}

\setlength{\epigraphwidth}{9.0cm}

\epigraph{{\it ``We must be clear that when it comes to atoms, language can be used only as in poetry. The poet, too, is not nearly so concerned with describing facts as with creating images and establishing mental connections."}

\flushright{--- Niels Bohr, {\it About describing atomic models\\ in the language of classical physics}}}

\section{Introduction}

The Rabinowitz action functional is a Lagrange multiplier functional which detects
periodic orbits on a fixed energy hypersurface that can go forward as well as backward
in time. The functional is invariant under the circle action obtained by reparametrisation
of loops. Therefore, one can consider its $S^1$-equivariant Rabinowitz Floer homology,
or more generally its Tate Rabinowitz Floer homology in which the generator of the cohomology
of $\mathbb{C}P^\infty$, the classifying space of the circle, is formally inverted.
According to~\cite{albers-cieliebak-frauenfelder}, the Rabinowitz action functional
on the loop space of the complex plane $\C$ can be used as a substitute for the non-existing
Tate classifying space of the circle. Here, in contrast to the above, the circle acts on
the loop space of $\C$ not on the domain but on the target, which gives rise
to a {\em free} action on the critical manifolds of the Rabinowitz action functional.

Already for a single harmonic oscillator, corresponding to the free loop space of $\C$, the Tate Rabinowitz Floer homology
turns out to be extremely rich. Since the Rabinowitz action functional on the free loop space of $\C$ is also used as Tate classifying space, it actually appears twice.
However, the two circle actions differ: once the circle acts on the domain, and once on the target. Although the harmonic oscillator has just one periodic orbit together
with all its integer iterates as in the case of Tate symplectic homology~\cite{albers-cieliebak-frauenfelder}, this richness arises from infinite staircases consisting of infinitely many iterates of the periodic orbit that are connected by gradient flow lines.
Do these staircases have some physical significance? Witten referred in his seminal paper~\cite{witten}, which greatly influenced Floer~\cite{floer}, to gradient flow lines as
\emph{instantons}. Can the critical points be interpreted as some states between which gradient flow lines
can tunnel?

In this paper, we examine the spectral numbers of homology classes in Tate Rabinowitz
Floer homology. It turns out that they are usually infinite, and only finite in very special
cases. Tate Rabinowitz Floer homology of $\C$ is endowed with a $\mathbb{Z}$-action
which gives it the structure of a~$\mathbb{Z}[\mathbb{Z}]$\nobreakdash-module. The $\mathbb{Z}$-action
does not preserve the spectral number but shifts it.
Given a~constant $\hbar>0$, we address the following question.

\begin{Question} Under which conditions can we associate to a $\mathbb{Z}[\mathbb{Z}]$-equivalence class of a nonvanishing homology class in Tate Rabinowitz Floer homology of $\C$
a spectral number in $S^1=\mathbb{R}/\hbar\mathbb{Z}$?
\end{Question}

As we shall see, this can only be done for a discrete set of energy values, thus giving
rise to some kind of quantized energy values. We refer to these
values as \emph{Tate Rabinowitz quantized energy values}. We then compute
them for the harmonic oscillator.

\begin{Theorem}\label{TheoremA}
The Tate Rabinowitz quantized energy values for
the harmonic oscillator are given by $\bigl(n-\tfrac{1}{2}\bigr)\hbar$, where $n \in \mathbb{N}$.
\end{Theorem}

These energy values coincide with the eigenvalues of the Schr\"odinger operator
for the harmonic oscillator \cite{schroedinger}. To obtain Theorem~\ref{TheoremA}, it is important
to define the spectral numbers with the help of the action value of the
Rabinowitz action functional corrected by a quarter of the Maslov index,
see (\ref{specnumb}) for the precise formula. This is reminiscent of the
semiclassical approach to quantization, where in Gutzwiller's famous trace formula
\cite{gutzwiller} a correction by a quarter of the Maslov index appears as well.

We next address the question of several harmonic oscillators,
which is of importance in semiclassics as an analytical test
of Gutzwiller's trace formula \cite{brack-jain, doll-ingold}. In fact,
the classical orbits in this case
are Lissajous curves, which are usually not periodic
unless one considers the isotropic case. For $m$ harmonic oscillators, each
harmonic oscillator can be reparametrised individually, so that the Rabinowitz action
functional is invariant under an action of the $m$-dimensional torus~$T^m$. The
corresponding Tate Rabinowitz Floer homology becomes a module over the group ring
$\mathbb{Z}[\mathbb{Z}^m]$. For the Tate Rabinowitz quantized energy values
in this case one obtains the following generalization of Theorem~\ref{TheoremA}.

\begin{Theorem}\label{TheoremB}
The Tate Rabinowitz quantized energy values for
$m$ harmonic oscillators of frequencies $\varpi_i$ for $1 \leq i \leq m$ are given by
\[
  E_n=\sum_{i=1}^m\bigl(n_i-\tfrac{1}{2}\bigr)\varpi_i\hbar
\]
where $n=(n_1,\dots,n_m) \in \mathbb{N}^m$.
\end{Theorem}

For $m=1$ and $\varpi_1=1$, Theorem~\ref{TheoremB}
reduces to Theorem~\ref{TheoremA}. Again, the Tate Rabinowitz quantized energy values
coincide with the eigenvalues of the Schr\"odinger operator and the semiclassical
predictions based on Gutzwiller's trace formula \cite{brack-jain, doll-ingold}.

It is interesting to compare the semiclassical approach with the approach via Tate
Rabinowitz homology. In both approaches periodic orbits play a crucial role. Besides
the Maslov index, in Gutzwiller's trace formula additional local invariants of periodic
orbits such as Floquet multipliers enter. Floer homology does not really see the Floquet
multipliers. In fact, the classical Morse lemma \cite{milnor} tells us that locally around
a critical point the only invariants are its action and its index. On the other hand,
Tate Rabinowitz Floer homology involves instantons, which do not play a role in
Gutzwiller's trace formula. From a more global perspective, it could be possible
to recover some Floquet multipliers from Floer theoretic data. For instance, in the case
of elliptic periodic orbits,
index iteration shows that the Floquet multipliers correspond to the mean index
\cite{long}. As for the mean index, one does not really need a flow but just a Fredholm
theory \cite{frauenfelder-weber}, it should be possible to
define Tate Rabinowitz quantized energy values also for Hamiltonian
delay equations. This is of special interest for the semiclassical study of the helium
atom \cite{cieliebak-frauenfelder-volkov}.

Theorem~\ref{TheoremA} is proved at the end of Section~\ref{harm}, and Theorem~\ref{TheoremB} at
the end of Section~\ref{sevharm}.

\section{The Rabinowitz action functional}

Let $(M,\omega)$ be a symplectic manifold. For simplicity, we assume that
$\omega$ is exact, i.e., $\omega={\rm d}\lambda$ for a one-form~$\lambda$. With the circle
$S^1=\mathbb{R}/\mathbb{Z}$, we denote by
\[
  \mathcal{L}:=\mathcal{L}_M:=C^\infty\bigl(S^1,M  \bigr)
\]
the free loop space of $M$. Given a smooth function
$H \colon M \to \mathbb{R}$, the Rabinowitz action functional is defined by
\[
\mathcal{A}^H \colon \ \mathcal{L}\times \mathbb{R}\to\R, \qquad
(v,\tau) \mapsto \int_{S^1} v^*\lambda-\tau \int_0^1 H(v(t)){\rm d}t.
\]
Although Moser declared this functional as ``certainly not suitable for an existence proof''~\cite{moser}, it was soon afterwards used by Rabinowitz to prove existence of periodic orbits in
his groundbreaking work \cite{rabinowitz}. One may think of this functional as the
Lagrange multiplier functional of the area functional for the constraint given by the mean value
of the Hamiltonian~\cite{cieliebak}. If we define the Hamiltonian
vector field of $H$ implicitly by the condition
\[
\mathrm{d}H=\omega(\cdot,X_H),
\]
then the critical points of $\mathcal{A}^H$ are solutions of the problem
\begin{equation*}
\begin{cases}
\partial_t v(t)=\tau X_H(v(t)),\\
H(v(t))=0,
\end{cases}\qquad t \in S^1.
\end{equation*}
For $\tau>0$, this corresponds after reparametrization to a periodic orbit
of period $\tau$ on the energy hypersurface $\Sigma=H^{-1}(0)$. For
$\tau=0$, the orbit is a constant point on $\Sigma$, and for $\tau<0$
the orbit is travelled backwards and has period $-\tau$.
This reminds one
in a mysterious way of the~Feynman--Stueckelberg interpretation of a positron being an electron
moving backwards in time~\cite{feynman,stueckelberg}. Since periodic orbits
can also go backward in time, the action is in general neither bounded from above nor below.
Algebraically, this has the effect that one can take different
completions in order to define various flavours of
Rabinowitz Floer homology \cite{cieliebak-frauenfelder2}.

The Rabinowitz action functional is invariant under the
reparametrization $S^1$ action
\begin{equation}\label{domact}
S^1 \times \mathcal{L} \times \mathbb{R} \to \mathcal{L} \times \mathbb{R},
\qquad (\rho,v,\tau) \mapsto (\rho_*v,\tau),
\end{equation}
where for $\rho \in S^1$ and $v \in \mathcal{L}$ the reparametrized loop is defined by
\[
\rho_*v(t)=v(t+\rho), \qquad t \in S^1.
\]
Moreover, it is antiinvariant under the involution
\[
I \colon \ \mathcal{L} \times \mathbb{R} \to \mathcal{L} \times \mathbb{R}, \qquad
(v,\tau) \mapsto (v^-,-\tau),
\]
where $v^-$ is obtained from $v$ by traveling backwards,
\[
v^-(t)=v(-t), \qquad t \in S^1.
\]
Here antiinvariance means that
\[
\mathcal{A}^H \circ I=-\mathcal{A}^H.
\]
There are two grading conventions on Rabinowitz Floer homology: one half-integer valued
 \cite{cieliebak-frauenfelder1} for which the grading is also
 antiinvariant under the involution $I$, and
one integer valued \cite{cieliebak-frauenfelder-oancea} which is more
in accordance with symplectic homology.
We will use the latter one.

\section{The Tate Rabinowitz action functional}

A special instance is the case $M=\mathbb{C}$ with its standard symplectic form
\[
\omega={\rm d}x \wedge {\rm d}y,
\]
$z=x+{\rm i}y \in \mathbb{C}$, and its standard primitive
\[
\lambda=\frac{1}{2}(x{\rm d}y-y{\rm d}x).
\]
For $\kappa>0$, we consider the Hamiltonian
\[
\mu_\kappa \colon \ \mathbb{C} \to \mathbb{R}, \qquad z \mapsto \pi |z|^2-\kappa.
\]
Its Hamiltonian vector field generates the standard circle action
\[
S^1 \times \mathbb{C} \to \mathbb{C}, \qquad (r,z) \mapsto {\rm e}^{2\pi {\rm i} r}z
\]
on the complex plane.
For $\ell \in \mathbb{Z}$, we define $w^\ell \in \mathcal{L}_{\mathbb{C}}$
by
\begin{equation}\label{well}
w^\ell (t)=\sqrt{\frac{\kappa}{\pi}}{\rm e}^{2\pi {\rm i} t \ell}, \qquad t \in S^1.
\end{equation}
Then we have a natural bijection
\begin{equation}\label{bijec}
\mathbb{Z} \times S^1 \cong
\operatorname{crit}\mathcal{A}^{\mu_\kappa},\qquad (\ell,r) \mapsto \bigl({\rm e}^{2\pi {\rm i}r }w^\ell, \ell\bigr).
\end{equation}
The action of a critical point is given by
\[
\mathcal{A}^{\mu_\kappa}\bigl({\rm e}^{2\pi {\rm i} r}w^\ell, \ell\bigr)=\kappa \ell.
\]
In the following, we consider $\mathcal{L}_\mathbb{C} \times \mathbb{R}$
endowed with the circle action on the target of a loop,
\begin{equation}\label{taract}
S^1 \times \mathcal{L}_\mathbb{C} \times \mathbb{R}
\to \mathcal{L}_\mathbb{C} \times \mathbb{R}, \qquad
(r,w,\eta) \mapsto \bigl({\rm e}^{2\pi {\rm i} r}w,\eta\bigr),
\end{equation}
in contrast to the circle action (\ref{domact}) for which the circle acts on the
domain.

For a manifold $M$, we define the
\emph{Tate space} of $M$ as the quotient
\[
\mathfrak{T}:=\mathfrak{T}_M:=(\mathcal{L}_M \times
\mathbb{R})\times_{S^1}(\mathcal{L}_\mathbb{C}\times \mathbb{R}),
\]
where we divide out the diagonal circle action given by
\begin{align*}
S^1 \times (\mathcal{L}_M \times \mathbb{R})\times(\mathcal{L}_\mathbb{C}\times \mathbb{R}) &\to  (\mathcal{L}_M \times \mathbb{R})\times(\mathcal{L}_\mathbb{C}\times \mathbb{R}),\\
(\rho,v,\tau,w,\eta) &\mapsto  \bigl(\rho_*v,\tau,{\rm e}^{2\pi {\rm i}\rho}w,\eta\bigr).
\end{align*}
Note that while the circle acts on the loop space of $M$ on the domain, it acts
on the loop space of $\mathbb{C}$ on the target.

For an exact symplectic manifold $(M,\omega={\rm d}\lambda)$ and a smooth
function $H\colon M \to \mathbb{R}$, the \emph{Tate Rabinowitz action
functional} for $H$ with respect to $\kappa>0$ is defined as
\[
\mathfrak{A}^H_\kappa \colon \ \mathfrak{T} \to \mathbb{R}, \qquad
[v,\tau,w,\eta] \to \mathcal{A}^H(v,\tau)+\mathcal{A}^{\mu_\kappa}(w,\eta).
\]
This is well defined because both $\mathcal{A}^H$ and $\mathcal{A}^{\mu_\kappa}$ are invariant
under the circle actions.

\section{The harmonic oscillator}\label{harm}

We consider the special case of the symplectic manifold $M=\mathbb{C}$ endowed
with its standard symplectic structure. In this case the Tate space is
\[
\mathfrak{T}_\mathbb{C}=(\mathcal{L}_\mathbb{C} \times \mathbb{R})
\times_{S^1} (\mathcal{L}_\mathbb{C}\times \mathbb{R}).
\]
Recall that the circle actions on the two loop spaces of $\mathbb{C}$ differ:
in the first component the circle acts on the domain, whereas in the second component it
acts on the target. For $\kappa>0$ and $E>0$ we consider the Tate Rabinowitz action functional
\[
\mathfrak{A}^E_\kappa:=\mathfrak{A}^{\mu_E}_\kappa \colon \
\mathfrak{T}_\mathbb{C} \to \mathbb{R}.
\]
In analogy to $w^\ell \in \mathcal{L}_\mathbb{C}$ in (\ref{well}),
we define for $\nu \in \mathbb{Z}$ the loops $v^\nu \in
\mathcal{L}_\mathbb{C}$ by
\[
v^\nu(t)=\sqrt{\frac{E}{\pi}}{\rm e}^{2\pi {\rm i} t \nu}, \qquad t \in S^1.
\]
As in (\ref{bijec}), we have a natural bijection
\[
\mathbb{Z} \times S^1 \times \mathbb{Z} \times S^1
\cong \operatorname{crit}\mathcal{A}^{\mu_E} \times \operatorname{crit}\mathcal{A}^{\mu_\kappa}
=\operatorname{crit}(\mathcal{A}^{\mu_E} \oplus \mathcal{A}^{\mu_\kappa})
\]
given by
\[
(\nu,s,\ell,r) \mapsto \bigl(\bigl({\rm e}^{2\pi {\rm i} s}v^\nu,\nu\bigr),({\rm e}^{2\pi {\rm i} r} w^\ell,\ell)\bigr).
\]
This gives rise to a bijection
\[
\mathbb{Z} \times \mathbb{Z} \times S^1=
\bigl(\mathbb{Z} \times S^1\bigr) \times_{S^1} \bigl(\mathbb{Z} \times S^1\bigr) =
\operatorname{crit}\mathcal{A}^{\mu_E} \times_{S^1}
\operatorname{crit}\mathcal{A}^{\mu_\kappa}=\operatorname{crit}\mathfrak{A}^E_\kappa.
\]
In view of
\begin{align*}
\rho_*\bigl({\rm e}^{2\pi {\rm i}s} v^\nu\bigr)(t) = \sqrt{\frac{E}{\pi}}{\rm e}^{2\pi {\rm i}s} {\rm e}^{
2\pi {\rm i}(t+\rho)\nu}
   = \sqrt{\frac{E}{\pi}}{\rm e}^{2\pi {\rm i}(s+\rho \nu)}{\rm e}^{2\pi {\rm i}t \nu}
 = {\rm e}^{2\pi {\rm i}(s+\rho \nu)} v^\nu(t),
\end{align*}
the circle action on $\mathbb{Z} \times S^1 \times \mathbb{Z} \times S^1$
is given by
\[
S^1 \times \mathbb{Z} \times S^1 \times \mathbb{Z} \times S^1
\to \mathbb{Z} \times S^1 \times \mathbb{Z} \times S^1,
\qquad (\rho,\nu,s,\ell,r) \mapsto (\nu, s+\rho \nu, \ell, r+\rho).
\]
The Tate Rabinowitz action functional $\mathfrak{A}^E_\kappa$ is Morse--Bott
and its critical manifold consists of a $\mathbb{Z} \times \mathbb{Z}$-family
of circles. We choose on each circle an auxiliary Morse function having one
maximum and one minimum. For $(\nu,\ell) \in \mathbb{Z} \times \mathbb{Z}$,
we denote by $u^+_{\nu,\ell}$ the point of maximum and by
\smash{$u^-_{\nu,\ell}$} the point of minimum.

Let $\partial$ be the boundary operator algebraically counting
cascades consisting of negative gradient flow lines of
$\mathfrak{A}^E_\kappa$ and negative gradient flow lines of the
auxiliary Morse function on the critical circles between the generators
$u^\pm_{\nu,\ell}$, see~\cite[Appendix A]{cieliebak-frauenfelder1}.
The two Rabinowitz action functionals give rise to a double filtration
\[
\mathcal{A}^{\mu_E}(u^+_{\nu,\ell})=\mathcal{A}^{\mu_E}(u^-_{\nu,\ell})=E\nu,
\qquad \mathcal{A}^{\mu_\kappa}(u^+_{\nu,\ell})=\mathcal{A}^{\mu_\kappa}(u^-_{\nu,\ell})
=\kappa \ell,
\]
both of which are nonincreasing along $\partial$. We further define a $\mathbb{Z}$-grading by{\samepage
\[
|u^-_{\nu,\ell}|=2(\nu+\ell), \qquad |u^+_{\nu,\ell}|=2(\nu+\ell)+1,
\]
so that $\partial$ decreases the grading by one.}

The boundary operator $\partial$
has been computed in~\cite[Section~3.2]{albers-cieliebak-frauenfelder} to be
\begin{equation}\label{parti}
\partial u^-_{\nu,\ell}=u^+_{\nu-1,\ell}+\nu u^+_{\nu,\ell-1}, \qquad
\partial u^+_{\nu,\ell}=0, \qquad (\nu,\ell) \in \mathbb{Z} \times \mathbb{Z}.
\end{equation}
This can be understood as follows. For degree reasons, the only
generator appearing in $\partial u^+_{\nu,\ell}$ is~$u^-_{\nu,\ell}$;
the coefficient $\langle\partial u^+_{\nu,\ell},u^-_{\nu,\ell}\rangle$
corresponds to the boundary operator for the circle and therefore
vanishes.
The only generators appearing in $\partial u^-_{\nu,\ell}$ are
$u^+_{\nu-1,\ell}$ and $u^+_{\nu,\ell-1}$.
Since the $S^1$-action is free on the second component,
the coefficient $\langle\partial u^-_{\nu,\ell},u^+_{\nu-1,\ell}\rangle$
corresponds to the boundary operator for the non-equivariant
Rabinowitz Floer chain complex of $\mathcal{A}^{\mu_E}$.
As the circle in the complex plane is displaceable,
the non-equivariant Rabinowitz Floer homology of
$\mathcal{A}^{\mu_E}$ vanishes~\cite{cieliebak-frauenfelder1},
so we must have ${\langle\partial u^-_{\nu,\ell},u^+_{\nu-1,\ell}\rangle=1}$.
The coefficient $\langle\partial u^-_{\nu,\ell},u^+_{\nu,\ell-1}\rangle$
corresponds to the boundary operator for the local $S^1$-equivariant
Rabinowitz Floer chain complex of the orbit~$v^\nu$.
Since the orbit $v^\nu$ rotates $\nu$ times under the circle
action on the domain, the local differential equals the
one on the doubly-infinite lens space $S^\infty_\infty/\Z_\nu$ given by
$\langle\partial u^-_{\nu,\ell},u^+_{\nu,\ell-1}\rangle=\nu$.

For $\nu \in \mathbb{Z}$, we denote by
$C_{2\nu}=C_{2\nu}\bigl(\mathfrak{A}^E_\kappa\bigr)$
the abelian group consisting of doubly infinite formal sums
of generators of degree $2\nu$, namely
\[
\xi=\sum_{\ell \in \mathbb{Z}} \xi_\ell u^-_{\nu-\ell,\ell}
\]
with coefficients $\xi_\ell \in \mathbb{Z}$.
Similarly, the abelian group
$
C_{2\nu+1}=C_{2\nu+1}\bigl(\mathfrak{A}^E_\kappa\bigr)$
consists of doubly infinite formal sums
\[
\xi=\sum_{\ell \in \mathbb{Z}} \xi_\ell u^+_{\nu-\ell,\ell}.
\]
We linearly extend $\partial$ to a map
$\partial \colon C_k \to C_{k-1}$, $k \in \mathbb{Z}$.
This is a boundary operator, i.e.,
$\partial^2=0$,
so that we can define the \emph{Tate Rabinowitz Floer homology} of the
harmonic oscillator as the quotient
\[
H_*:=H_*\bigl(\mathfrak{A}^E_\kappa\bigr):= \ker\partial/\mathrm{im}\,\partial.
\]
We further have linear isomorphisms
\begin{equation}\label{umap}
\mathfrak{U} \colon C_k \stackrel{\cong}\longrightarrow C_{k+2}
\end{equation}
that on generators are given by
\[
\mathfrak{U}(u^+_{\nu,\ell})=u^+_{\nu,\ell+1}, \qquad
\mathfrak{U}(u^-_{\nu,\ell})=u^-_{\nu,\ell+1}.
\]
These maps commute with the boundary operator,
$\mathfrak{U} \circ \partial=\partial \circ \mathfrak{U}$,
so they induce isomorphisms on homology
$\mathfrak{u} \colon  H_k \stackrel{\cong}\longrightarrow H_{k+2}$, $k \in \mathbb{Z}$.
This endows the homology $H_*$ with the structure of a
$\mathbb{Z}[\mathbb{Z}]$-module, where the formal variable $\mathfrak{u}$ has
degree two.

Our next goal is to show that the Tate Rabinowitz Floer homology of
the harmonic oscillator is nontrivial. However, our first lemma looks
like bad news. We say that
$\xi \in C_k$ has \emph{finite positive support} if there exists $L \in
\mathbb{Z}$ such that the coefficients $\xi_\ell$ vanish for $\ell>L$.
Note that all elements of odd degree are cycles.
\begin{Lemma}\label{nonhomo}
If $\eta \in C_{2\nu-1}$ has finite positive support, then $\eta \in \mathrm{im}\,\partial$.
\end{Lemma}
\begin{proof} We write
\[
\eta=\sum_{\ell \in \mathbb{Z}} \eta_\ell u^+_{\nu-\ell-1,\ell}.
\]
We can assume without loss of generality that $\eta \neq 0$.
Hence, since $\eta$ has finite positive support, there exists $L \in \mathbb{Z}$ such that
\[
L=\max\{\ell\mid \eta_\ell\neq 0\}.
\]
We need to construct
\[
\xi=\sum_{\ell \in \mathbb{Z}} \xi_\ell u^-_{\nu-\ell,\ell} \in C_{2\nu}
\]
such that
$\partial \xi=\eta$.
We set
$\xi_\ell=0$, $\ell > L$.
For $\ell \leq L$, we define downward recursively
\begin{equation}\label{rec}
\xi_\ell=\eta_\ell-(\nu-\ell-1)\xi_{\ell+1}.
\end{equation}
Since both $\eta_\ell$ and $\xi_\ell$ vanish for $\ell>L$, the formula
(\ref{rec}) actually holds true for all $\ell \in \mathbb{Z}$.
Using~(\ref{parti}) and~(\ref{rec}), we compute
\begin{align*}
\partial \xi&= \sum_{\ell \in \mathbb{Z}}\xi_\ell
\partial u^-_{\nu-\ell,\ell}
 = \sum_{\ell \in \mathbb{Z}}\xi_\ell\bigl(u^+_{\nu-\ell-1,\ell}+(\nu-\ell)u^+_{\nu-\ell,\ell-1}\bigr) \\
&= \sum_{\ell \in \mathbb{Z}}\bigl(\xi_\ell+\xi_{\ell+1}(\nu-\ell-1)\bigr)u^+_{\nu-\ell-1,\ell}
 = \sum_{\ell \in \mathbb{Z}}\eta_\ell u^+_{\nu-\ell-1,\ell}
 = \eta.
\end{align*}
This proves the lemma.
\end{proof}

We say that $\xi \in C_k$ has \emph{infinite positive support} if it does not have
finite positive support. We further say that $\xi$ is \emph{positively bounded}
if there exists $N \in \mathbb{N}$ such that
$|\xi_\ell| \leq N$, $\ell \in \mathbb{N}$.
Note that we impose this condition only on the coefficients of
positive index $\ell$, so the coefficients
can be unbounded as $\ell$ tends to minus infinity. The next lemma guarantees that
our Tate Rabinowitz Floer homology is nontrivial.
\begin{Lemma}\label{homo}
Assume that $\eta \in C_{2\nu-1}$ has infinite positive support but is positively bounded.
Then $\eta \notin \mathrm{im}\,\partial$.
\end{Lemma}
\begin{proof}
Arguing by contradiction, suppose that
\[
\eta=\sum_{\ell \in \mathbb{Z}} \eta_\ell u^+_{\nu-\ell-1,\ell}
\]
admits a primitive, i.e., there exists
\[
\xi=\sum_{\ell \in \mathbb{Z}} \xi_\ell u^-_{\nu-\ell,\ell} \in C_{2\nu}
\]
such that
$\partial \xi=\eta$.
The coefficients of $\xi$ have to satisfy the recursion relation~(\ref{rec}).
Since $\eta$ is positively bounded, there exists $N \in \mathbb{N}$ such that
$|\eta_\ell| \leq N$, $\ell \geq \nu$.
Hence we obtain from~\eqref{rec} the estimate
\begin{equation*}
|\xi_{\ell+1}| \leq \frac{|\xi_\ell|+N}{\ell+1-\nu}, \qquad \ell \geq \nu.
\end{equation*}
We claim that this estimate implies $\lim_{\ell\to\infty}\xi_\ell=0$.
To see this, define $x_k:=|\xi_{k+\nu}|$ for $k\geq 0$. Then $x_{k+1}\leq\frac{x_k+N}{k+1}$, which by induction implies
\[
  x_k\leq\frac{x_0}{k!}+Nc_k,\qquad c_k:=\frac{1}{k}+\frac{1}{k(k-1)}+\cdots+\frac{1}{k!}.
\]
Since $c_k$ is a sum of $k$ terms each $\leq 1/k$, we have $c_k\leq 1$ for all $k$. This implies $c_k=\frac{1+c_{k-1}}{k}\leq 2/k$, and therefore the claim in view of
\[
  x_k \leq \frac{x_0}{k!} + \frac{2 N}{k}\xrightarrow[k\to\infty]{} 0.
\]
Since the $\xi_\ell$ are integers, the claim implies that there exists an $L$ such that $\xi_\ell=0$ for all $\ell\geq L$. By~\eqref{rec}, this yields $\eta_\ell=0$ for all $\ell\geq L$, contradicting the hypothesis that $\eta$ has infinite positive support.
\end{proof}

While the Tate Rabinowitz homology for the harmonic oscillator is extremely rich
in odd degrees, we show next that it vanishes in even degrees.

\begin{Lemma}\label{noed}
For every $\nu \in \mathbb{Z}$, we have $\ker(\partial\colon C_{2\nu}\to
C_{2\nu-1})=\{0\}$, hence $H_{2\nu}=\{0\}$.
\end{Lemma}
\begin{proof} Suppose that
\[
\xi=\sum_{\ell \in \mathbb{Z}}\xi_\ell u^-_{\nu-\ell,\ell}
\in C_{2\nu}
\]
satisfies
\[
\partial \xi=0.
\]
We will show that this implies $\xi=0$. Since $\xi$ lies in the
kernel of the boundary operator, we obtain as a special instance of
(\ref{rec}) the following recursion relation between
its coefficients:
\begin{equation}\label{recsimp}
\xi_\ell=(\ell+1-\nu)\xi_{\ell+1}, \qquad \ell \in \mathbb{Z}.
\end{equation}
In particular, plugging $\ell=\nu-1$ into this formula we obtain
$\xi_{\nu-1}=0$,
and therefore by downward induction using (\ref{recsimp})
\begin{equation}\label{van1}
\xi_\ell=0, \qquad \ell \leq \nu-1.
\end{equation}
It remains to show that the coefficients vanish also for $\ell\geq \nu$.
In order to prove this, we first show the following claim.

\begin{Claim}For every $\ell_0 \in \mathbb{Z}$, there exists $\ell \geq \ell_0$
such that $\xi_\ell=0$.
\end{Claim}

Arguing by contradiction, suppose there exists $\ell_0$ such that
\[
\xi_\ell \neq 0, \qquad \ell \geq \ell_0.
\]
Using (\ref{recsimp}) inductively, we obtain
\[
\xi_{\nu+k}=\frac{\xi_\nu}{k!} \xrightarrow[k\to\infty]{} 0.
\]
Since $\xi_{\nu+k}\in\Z\setminus\{0\}$ for all $k\geq 0$,
this is a contradiction proving the claim.

If $\xi_L=0$ for some $L\geq \nu$, then downward induction using (\ref{recsimp}) shows that
$\xi_\ell=0$ for every $\nu \leq \ell \leq L$. Since by the claim we find
$L$ arbitrarily large such that $\xi_L=0$, we deduce that
$\xi_\ell=0$, $\ell \geq \nu$.
Combining this with (\ref{van1}), we see that all coefficients of $\xi$
vanish, so that $\xi=0$. This proves the lemma.
\end{proof}

It follows from Lemma~\ref{homo} that $H_{2\nu+1} \neq \{0\}$ for every $\nu \in \mathbb{Z}$.
We next associate to a nonzero homology class $\alpha \in H_{2\nu+1}$ a spectral number.
For this, we introduce on generators the {\em shifted action}
\[
  \wt{\mathfrak{A}}^E_\kappa(u^+_{\nu,\ell}) :=
  \mathfrak{A}^E_\kappa(u^+_{\nu,\ell}) + \frac{\nu\hbar}{2}.
\]
Since the Conley--Zehnder index of the first component $v^\nu$ of
$u_{\nu,\ell}=(v^\nu,w^\ell)$ is $2\nu$, this corresponds
to the addition of $\hbar/4$ times the Conley--Zehnder index.
This is reminiscent of the rule in semiclassics of adding $\hbar/4$
times the Maslov index \cite{gutzwiller}.
For a mathematical discussion of the Maslov index
in semiclassics, we refer to the paper by Sun \cite{sun}.

Using this, we first define for a nonvanishing chain
\[
\xi=\sum_{\ell \in \mathbb{Z}}\xi_\ell u^+_{\nu-\ell,\ell} \in
C_{2\nu+1}
\]
the quantity
\begin{equation}\label{specnumb}
\sigma(\xi):=\sup\bigl\{\wt{\mathfrak{A}}^E_\kappa(u^+_{\nu-\ell,\ell}) \mid \xi_\ell \neq 0\bigr\}
\in (-\infty,\infty],
\end{equation}
and then we define for a homology class $0\neq\alpha \in H_{2\nu+1}$
its {\em spectral number} as the infimum over all its
representatives,
\[
\sigma(\alpha):=\inf \{\sigma(\xi)\mid[\xi]=\alpha \} \in [-\infty,\infty].
\]

\begin{Proposition}\label{spec}
The spectral number of each $0\neq\alpha \in H_{2\nu+1}$ satisfies
\[
\sigma(\alpha)= \begin{cases}
\infty, & E<\kappa-\frac{\hbar}{2},\\
\bigl(E+\frac{\hbar}{2}\bigr)\nu, & E=\kappa-\frac{\hbar}{2},\\
-\infty, & E>\kappa-\frac{\hbar}{2}.
\end{cases}
\]
\end{Proposition}
\begin{proof}
Note that
\begin{align}
  \wt{\mathfrak{A}}^E_\kappa(u^+_{\nu-\ell,\ell})
  &= \mathfrak{A}^E_\kappa(u^+_{\nu-\ell,\ell}) + \tfrac{(\nu-\ell)\hbar}{2}
  = E(\nu-\ell)+\kappa\ell+\tfrac{(\nu-\ell)\hbar}{2} \nonumber\\
  &= \bigl(E+\tfrac{\hbar}{2}\bigr)\nu+\bigl(\kappa-E-\tfrac{\hbar}{2}\bigr)\ell.\label{act}
\end{align}
We now discuss the three cases one by one.

\textit{Case $E<\kappa-\tfrac{\hbar}{2}$:}
Consider $\xi \in C_{2\nu+1}$ with $\alpha=[\xi]$. Since $\alpha\neq 0$,
by Lemma~\ref{nonhomo} we deduce that $\xi$ has infinite positive support.
Since $\kappa>E+\tfrac{\hbar}{2}$, we infer from (\ref{act}) that
$\sigma(\xi)=\infty$.
Since~$\xi$ was an arbitrary representative of $\alpha$, this implies that
$\sigma(\alpha)=\infty$.

\textit{Case $E=\kappa-\tfrac{\hbar}{2}$:}
In this case we have
$\wt{\mathfrak{A}}^E_\kappa(u^+_{\nu-\ell,\ell})=\bigl(E+\frac{\hbar}{2}\bigr)\nu$
independently of $\ell$, which immediately implies that
$\sigma(\alpha)=\bigl(E+\tfrac{\hbar}{2}\bigr)\nu$.

\textit{Case $E>\kappa-\tfrac{\hbar}{2}$:}
Consider
\[
\xi=\sum_{\ell \in \mathbb{Z}}\xi_\ell u^+_{\nu-\ell,\ell} \in C_{2\nu+1}
\]
with $[\xi]=\alpha$.
For $L \in \mathbb{N}$, abbreviate
\[
\xi^L:=\sum_{\ell \geq L}\xi_\ell u^+_{\nu-\ell,\ell}.
\]
Then $\xi-\xi^L$ has finite positive support, so by Lemma~\ref{nonhomo}
we have
$\xi-\xi^L \in \mathrm{im}\,\partial$,
and therefore
$
\bigl[\xi^L\bigr]=[\xi]=\alpha$.
Since $\alpha \neq 0$, the quantity
\[
\ell^L:=\min\{\ell \geq L\mid \xi_\ell \neq 0\}
\]
is well defined. Since $\kappa<E+\tfrac{\hbar}{2}$, the function
\[
\mathbb{Z} \to \mathbb{R}, \qquad \ell \mapsto
  \bigl(E+\tfrac{\hbar}{2}\bigr)\nu+\bigl(\kappa-E-\tfrac{\hbar}{2}\bigr)\ell
\]
is strictly decreasing, so that by (\ref{act}) we have
\[
\sigma\bigl(\xi^L\bigr)=
  \bigl(E+\tfrac{\hbar}{2}\bigr)\nu+\bigl(\kappa-E-\tfrac{\hbar}{2}\bigr)\ell^L \leq
  \bigl(E+\tfrac{\hbar}{2}\bigr)\nu+\bigl(\kappa-E-\tfrac{\hbar}{2}\bigr)L.
\]
Again using $\kappa<E+\tfrac{\hbar}{2}$, we deduce that
\[
\lim_{L \to \infty} \sigma\bigl(\xi^L\bigr)=-\infty.
\]
Since each $\xi^L$ represents the homology class $\alpha$, this implies that
$\sigma(\alpha)=-\infty$.
This proves the proposition.
\end{proof}

As explained above, the $\mathfrak{U}$-map (\ref{umap}) endows
Tate Rabinowitz Floer homology with the structure of a $\mathbb{Z}[\mathbb{Z}]$-module.
The $\mathfrak{U}$-map does not preserve the action, but in view of
\[
\wt{\mathfrak{A}}^E_\kappa\bigl(\mathfrak{U}(u^\pm_{\nu,\ell})\bigr)=
  \wt{\mathfrak{A}}^E_\kappa(u^\pm_{\nu,\ell+1})=
  \wt{\mathfrak{A}}^E_\kappa(u^\pm_{\nu,\ell})+
\kappa
\]
it shifts it uniformly by the amount of $\kappa$, so that for $0\neq\alpha \in H_*$
we have
\begin{equation}\label{shift}
\sigma(\mathfrak{u}\alpha)=\sigma(\alpha)+\kappa.
\end{equation}
We are now in a~position to prove Theorem~\ref{TheoremA} from the introduction.

\begin{proof}[Proof of Theorem~\ref{TheoremA}]
By the shift property~\eqref{shift}, we can associate to a
$\Z[\Z]$-equivalence class (under the action of $\mathfrak{u}$) a
spectral number in $\R/\hbar\Z$ only if $\kappa\in\hbar\Z$, i.e.,~$\kappa\in\hbar\N$
because $\kappa>0$.
By Lemma~\ref{noed}, there are no spectral numbers for $\alpha\in
H_{2\nu}=\{0\}$. By Proposition~\ref{spec}, the spectral number
$\sigma(\alpha)$ of $0\neq\alpha\in H_{2\nu+1}$ is finite if and
only if $E=\kappa-\tfrac{\hbar}{2}\in\bigl(\N-\tfrac12\bigr)\hbar$.
\end{proof}

\begin{Remark}
Lemmas~\ref{homo} and~\ref{noed} rely crucially on the fact that the
coefficients are $\Z$, and they are not true with rational coefficients.
A similar phenomenon occurs in~\cite{albers-cieliebak-frauenfelder}.
One may wonder whether this has some number theoretic significance,
cf.~\cite{cieliebak-frauenfelder3} for a relation between spectra and
number theory.
\end{Remark}

\section{The iterated Tate Rabinowitz action functional}

Consider exact symplectic manifolds $(M_i,\omega_i={\rm d} \lambda_i)$,
$1 \leq i \leq m$,
and smooth functions
\[
H_i \colon \ M_i \to \mathbb{R},\qquad 1 \leq i \leq m.
\]
The product
\begin{equation*}
(M,\omega)=\biggl(\prod_{i=1}^m M_i, \bigoplus_{i=1}^m \omega_i\biggr)
\end{equation*}
is an exact symplectic manifold with primitive
\[
\lambda=\bigoplus_{i=1}^m \lambda_i.
\]
We define the smooth function
\[
H \colon \ M \to \mathbb{R}, \qquad (x_1, \dots,x_m) \mapsto
\sum_{i=1}^m H_i(x_i).
\]
The $m$-dimensional torus
\[
T^m=\underbrace{S^1\times \dots \times S^1}_{m\,\,\textrm{times}}
\]
acts on the free loop space $\mathcal{L}_M$ as follows.
For $\rho=(\rho_1,\dots,\rho_m) \in T^m$
and $v=(v_1,\dots,v_m)\in \mathcal{L}_M$, define
\[
\rho_* v=\bigl((\rho_1)_*v_1,\dots,(\rho_m)_*v_m\bigr),
\]
i.e., the action
$T^m \times \mathcal{L}_M \to \mathcal{L}_M$
reparametrizes all components of $v$ individually.
We define the \emph{iterated Tate space}
\[
\mathfrak{T}^m=\mathfrak{T}^m_M=(\mathcal{L}_M \times \mathbb{R})
\times_{T^m} (\mathcal{L}_\mathbb{C}\times \mathbb{R})^m,
\]
where $T^m$ acts on $(\mathcal{L}_\mathbb{C}\times \mathbb{R})^m$ componentwise
on the target of the loops as in (\ref{taract}). Let
$\kappa=(\kappa_1,\dots,\kappa_m) \in (0,\infty)^m$ be a vector with positive
entries. The \emph{$m$-fold iterated Rabinowitz action functional}
$\mathfrak{A}^H_\kappa \colon \mathfrak{T}^m \to \mathbb{R}$
is defined for $(v,\tau) \in \mathcal{L}_M \times \mathbb{R}$ and
$(w,\eta)=(w_1,\eta_1,\dots,w_m,\eta_m) \in (\mathcal{L}_\mathbb{C}\times \mathbb{R})^m$
as
\[
\mathfrak{A}^H_\kappa([v,\tau,w,\eta])=\mathcal{A}^H(v,\tau)
+\sum_{i=1}^m \mathcal{A}^{\mu_{\kappa_i}}(w_i,\eta_i).
\]

\section{Several uncoupled harmonic oscillators}\label{sevharm}

We consider the case where $M_i=\mathbb{C}$ for $1 \leq  i \leq m$, with
the iterated Tate space $\mathfrak{T}^m=\mathfrak{T}^m_{\mathbb{C}^m}$.
We~fix a vector
\[
\varpi=\bigl(\tfrac{1}{\varpi_1},\dots,\tfrac{1}{\varpi_m}\bigr) \in (0,\infty)^m
\]
and energy $E>0$ and define the smooth function
\[
\mu_{E,\varpi} \colon \ \mathbb{C}^m \to \mathbb{R}, \qquad
(z_1,\dots,z_m) \mapsto \sum_{i=1}^m \pi \varpi_i |z_i|^2-E.
\]
For an additional vector $\kappa=(\kappa_1,\dots,\kappa_m) \in (0,\infty)^m$, we
will study in this section the Floer complex for the $m$-fold iterated Rabinowitz
action functional
\[
\mathfrak{A}^E_\kappa=\mathfrak{A}^{E,\varpi}_\kappa=\mathfrak{A}^{\mu_{E,\varpi}}_\kappa
\colon \ \mathfrak{T}^m \to \mathbb{R}.
\]
In order to describe this complex, note that the sublevel set
$\{\mu_{E,\varpi}\leq 0\}$ of the above Hamiltonian is an
ellipsoid. We recall
from~\cite{floer-hofer-wysocki} its Floer complex.
Denote by \smash{$\bigl(\widetilde\Omega_\nu\bigr)_{\nu\in\N}$} the sequence of positive integer multiples of
$\tfrac{1}{\varpi_1}, \dots, \tfrac{1}{\varpi_m}$ arranged in
nondecreasing order with repetitions. They correspond to closed Reeb
orbits $\gamma_\nu$ of action \smash{$E\widetilde\Omega_\nu$} and Conley--Zehnder index $m-1+2\nu$
(understood in the Morse--Bott sense in case of rational dependencies
among the $\varpi_i$). The (non-equivariant) Floer complex has for
each $\gamma_\nu$ two generators $\gamma_\nu^\pm$, corresponding to the
maximum and minimum of a Morse function on the circle, such that
$\partial\gamma_\nu^+=0$ and $\partial\gamma_{\nu+1}^-=\gamma_\nu^+$.
(In the case of rational dependencies, this description also
holds, where the $\gamma_\nu$ are generators of the
homology of the critical manifolds.)

For the Rabinowitz Floer complex of $\mu_{E,\varpi}$, we need to
consider all integer multiples rather than just the positive ones.
To get a more symmetric description, we will shift the index~$\nu$ by~$\tfrac{m-1}{2}$.
Thus, we denote by $(\Omega_\nu)_{\nu\in
  \mathbb{Z}+\frac{m-1}{2}}$ the sequence of integer multiples of
$\tfrac{1}{\varpi_1}, \dots, \tfrac{1}{\varpi_m}$ arranged in
nondecreasing order with repetitions such that
\[
\Omega_{-\frac{m-1}{2}}=\Omega_\frac{m-1}{2}=0.
\]
Note that
\begin{equation}\label{ant1}
\Omega_\nu=-\Omega_{-\nu}, \qquad \nu \in \mathbb{Z}+\tfrac{m-1}{2}.
\end{equation}
They correspond to generators $\gamma_\nu^\pm$ of action
\[
  \mathfrak{A}^{\mu_{E,\varpi}}\bigl(\gamma_\nu^\pm\bigr) = E\Omega_\nu
\]
and degrees
\[
  |\gamma_\nu^-| = m-1+2(\nu-\tfrac{m-1}{2}) = 2\nu,\qquad
  |\gamma_\nu^+| = 2\nu + 1
\]
such that
\begin{equation}\label{diff-ellipsoid}
  \partial\gamma_\nu^+=0,\qquad
  \partial\gamma_\nu^-=\gamma_{\nu-1}^+,\qquad \nu\in\Z+\tfrac{m-1}{2}.
\end{equation}
We denote the corresponding generators of the Tate Rabinowitz Floer complex by
$u^+_{\nu,k}$ and $u^-_{\nu,k}$, where $\nu \in \mathbb{Z}+\tfrac{m-1}{2}$ and
$k \in \mathbb{Z}^m$.
By the preceding discussion, they have action
\begin{equation}\label{wirk}
\mathfrak{A}^E_\kappa\bigl(u^\pm_{\nu,k}\bigr)=\langle \kappa,k\rangle+E\Omega_\nu
\end{equation}
and degrees
\begin{align*}
  &|u^+_{\nu,k}|= 2\bigl(\nu+\zeta_k\bigr)+1 \in 2\mathbb{Z}+m\\
  &|u^-_{\nu,k}|= 2\bigl(\nu+\zeta_k\bigr) \in 2\mathbb{Z}+m-1,
\end{align*}
where for a vector $k \in \mathbb{Z}^m$ we denote by
\[
\zeta_k=\sum_{i=1}^m k_i
\]
the sum of its entries.
To describe the differential, we introduce some notation.
Let $\mathfrak{b}=\{e_1,\dots,e_m\}$ be the standard basis of
$\mathbb{R}^m$.
We choose maps
\[
\varepsilon\colon \ \mathbb{Z}+\tfrac{m-1}{2} \to \mathfrak{b}, \qquad
\mu \colon \ \mathbb{Z}+\tfrac{m-1}{2} \to \mathbb{Z}
\]
such that
\begin{equation}\label{ant2}
\Omega_\nu=\mu(\nu)\big\langle \varpi, \varepsilon
(\nu)\big\rangle, \qquad \nu \in \mathbb{Z}+\tfrac{m-1}{2}
\end{equation}
and
\begin{equation}\label{ant3}
\varepsilon(\nu)=\varepsilon(-\nu), \qquad \mu(\nu)=-\mu(-\nu), \qquad
\nu \in \mathbb{Z}+\tfrac{m-1}{2}.
\end{equation}
If the ratios $\tfrac{\varpi_i}{\varpi_j}$ are irrational for all
$i \neq j$, then the maps $\varepsilon$ and
$\mu$ for $|\nu|>\tfrac{m-1}{2}$ are uniquely determined by
(\ref{ant2}) and in this case (\ref{ant3}) follows from (\ref{ant1}).
In the presence of rational ratios the map $\mu$ is still unique,
corresponding to the multiplicity, but the map $\varepsilon$ is not
uniquely determined.

\begin{Lemma}
The boundary operator is given on generators by
\[
\partial u^-_{\nu,k}=u^+_{\nu-1,k}+\mu(\nu)u^+_{\nu,k-\varepsilon(\nu)},\qquad \partial u^+_{\nu,k}=0.
\]
\end{Lemma}
\begin{proof}
Vanishing of $\partial u^+_{\nu,k}$ and the first term in $\partial
u^-_{\nu,k}$ follow from~\eqref{diff-ellipsoid}.
For the second term in $\partial u^-_{\nu,k}$, let
$\varepsilon(\nu)=e_i$, so that the first
component of $\partial u^-_{\nu,k}$ corresponds to the $\mu(\nu)$-fold
covered closed Reeb orbit moving only in the $i$-th coordinate.
As discussed in~Section~\ref{harm}
(see~\cite[Section~3.2]{albers-cieliebak-frauenfelder}), the differential
becomes the one on the doubly-infinite lens space
$S^\infty_\infty/\Z_{\mu(\nu)}$ in the $i$-th component, and the
doubly-infinite projective space $\C P^\infty_\infty$ in the other
components. Since the differential on $\C P^\infty_\infty$ is trivial,
this proves the lemma.
\end{proof}

Note that, since $\Omega_\nu$ is nondecreasing in $\nu$, the action
$\mathfrak{A}^E_\kappa$ in~\eqref{wirk} is
nonincreasing along $\partial$.
Given $\nu \in \mathbb{Z}+\tfrac{m-1}{2}$, we define abelian groups
\[
C_{2\nu+1}=\bigg\{\sum_{\substack{k \in \mathbb{Z}^m}}\xi_k u^+_{\nu
-\zeta_k,k} \,\Bigl|\, \xi_k \in
\mathbb{Z}\bigg\}, \qquad C_{2\nu}=\bigg\{\sum_{\substack{k \in \mathbb{Z}^m}}
\xi_k u^-_{\nu-\zeta_k,k} \,\Bigl|\, \xi_k \in
\mathbb{Z}\bigg\}.
\]
To compute the homology of the chain complex $(C_*,\partial)$,
  we will construct a subcomplex from which the complex is obtained by
  the action of suitable $\mathfrak{U}$-maps.
We define
\[
\beta^+_\nu \colon \ \mathbb{Z} \to \mathbb{Z}^m
\]
recursively by the conditions
\[
\beta^+_\nu(0)=0, \qquad \beta^+_\nu(\ell-1)=\beta_\nu^+(\ell)-\varepsilon(\nu+1-\ell),
\quad \ell \in \mathbb{Z}.
\]
More explicitly, we have
\[
\beta^+_\nu(\ell) = \begin{cases}
  \displaystyle \sum_{j=1}^\ell\varepsilon(\nu+1-j), & \ell>0,\cr
 \displaystyle -\sum_{j=1}^{-\ell}\varepsilon(\nu+j), & \ell<0.
\end{cases}
\]
We further define
\[
\beta^-_\nu \colon \ \mathbb{Z} \to \mathbb{Z}^m,\qquad
\beta^-_\nu(\ell):=\beta^+_{\nu-1}(\ell).
\]
Note that
\begin{equation}\label{varsigma}
\zeta_{\beta^+_\nu(\ell)}=\zeta_{\beta^-_\nu(\ell)}=\ell,
\qquad \ell \in \mathbb{Z}.
\end{equation}
We define subgroups
$B_{2\nu} \subset C_{2\nu}$, $B_{2\nu+1} \subset C_{2\nu+1}$
by
\[
B_{2\nu+1}=\bigg\{\sum_{\ell \in \mathbb{Z}}\xi_\ell
u^+_{\nu-\ell,\beta^+_\nu(\ell)} \,\Bigl|\, \xi_\ell \in \mathbb{Z}\bigg\}, \qquad
B_{2\nu}=\bigg\{\sum_{\ell \in \mathbb{Z}}\xi_\ell
u^-_{\nu-\ell,\beta^-_\nu(\ell)} \,\Bigl|\, \xi_\ell \in \mathbb{Z}\bigg\}.
\]
For each $n \in \mathbb{Z}$, we can recover $C_n$ from $B_n$ using
suitable $\mathfrak{U}$-maps which we introduce next. For $1 \leq i
\leq m$, we define linear maps
$\mathfrak{U}_i \colon C_n \to C_{n+2}$
which are given on generators by
\[
\mathfrak{U}_i(u^\pm_{\nu,k})=u^\pm_{\nu,k+e_i}.
\]
Note that the maps $\mathfrak{U}_i$ commute with each other as well as with the
boundary operator $\partial$. In particular, the maps $\mathfrak{U}_i$ induce
commuting maps on homology
\[
\mathfrak{u}_i  \colon \ H_n \to H_{n+2}, \qquad 1 \leq i \leq m,
\]
so that $C_*$ as well as $H_*$ get endowed with the structure of a
$\mathbb{Z}[\mathbb{Z}^m]$-module. This action does not fix the
degree. However, if we restrict to the subgroup
\[
U:=\big\{k \in \mathbb{Z}^m \mid \zeta_k=0\big\} \cong \mathbb{Z}^{m-1},
\]
then $\mathbb{Z}[U]$ keeps $C_n$ invariant for every $n \in \mathbb{Z}$, so that
$C_n$ gets endowed with the structure of a~$\mathbb{Z}[U]$-module. We can now
express $C_n$ with the help of $B_n$ as
\[
C_n=\bigg\{\sum_{\mathfrak{U} \in U} \mathfrak{U} \xi_\mathfrak{U} \,\Bigl|\,
\xi_\mathfrak{U} \in B_n\bigg\}.
\]
From
\begin{align*}
\partial u^-_{\nu-\ell,\beta^-_\nu(\ell)}&= u^+_{\nu-\ell-1,\beta_\nu^-(\ell)}
+\mu(\nu-\ell)u^+_{\nu-\ell,\beta^-_\nu(\ell)-\varepsilon(\nu-\ell)}\\
&= u^+_{\nu-1-\ell,\beta_{\nu-1}^+(\ell)}
+\mu(\nu-\ell)u^+_{\nu-1-(\ell-1),\beta^+_{\nu-1}(\ell)-\varepsilon(\nu-\ell)}\\
&= u^+_{\nu-1-\ell,\beta_{\nu-1}^+(\ell)}
+\mu(\nu-\ell)u^+_{\nu-1-(\ell-1),\beta^+_{\nu-1}(\ell-1)},
\end{align*}
we see that
$\partial u^-_{\nu-\ell,\beta^-_\nu(\ell)} \in B_{2\nu-1}$,
so the boundary operator restricts to a map
\[
\partial^B:=\partial|_{B_n} \colon \ B_n \to B_{n-1}, \qquad n \in \mathbb{Z}.
\]
Denoting by $H_*^B$ the homology of the subcomplex $\bigl(B_*,\partial^B\bigr)$ of
$(C_*,\partial)$, we get
\[
H_n=\biggl\{\sum_{\mathfrak{u}\in U}\mathfrak{u}\alpha_\mathfrak{u} \,\Bigl|\,
\alpha_\mathfrak{u} \in H^B_n\biggr\}.
\]
As in the case of a single harmonic oscillator, we say that $\xi \in B_n$
has finite positive support if the coefficients $\xi_\ell$ vanish for
$\ell$ large enough, and otherwise we say that it has infinite positive
support. If the coefficients $\xi_\ell$ are bounded for positive indices
$\ell$ we say that $\xi$ is positively bounded. Since $|\mu(\nu)|$
is unbounded, Lemmas~\ref{nonhomo},~\ref{homo} and~\ref{noed}
carry over to the present setting and yield the following proposition.

\begin{Proposition}\label{torprop1}
For each $\nu \in \mathbb{Z}+\tfrac{m-1}{2}$, the following statements hold.
The group $H_{2\nu}^B$, and therefore also $H_{2\nu}$, is trivial.
If~$\xi \in B_{2\nu+1}$ has finite positive support, then it
lies in the image of $\partial^B$;
if~it has infinite positive support but is positively bounded, then it
does not lie in the image of~$\partial^B$. In particular, the group
$H_{2\nu+1}^B$, and therefore also $H_{2\nu+1}$, is nontrivial.
\end{Proposition}
As in~Section~\ref{harm}, we introduce on generators the shifted action
\[
  \wt{\mathfrak{A}}^E_\kappa(u^+_{\nu,k}) :=
  \mathfrak{A}^E_\kappa(u^+_{\nu,k}) + \frac{\nu\hbar}{2}.
\]
For a nonvanishing
\[
\xi=\sum_{k \in \mathbb{Z}^m}\xi_k u^+_{\nu-\zeta_k,k} \in C_{2\nu+1},
\]
we define the quantity
\[
\sigma(\xi):=\sup\bigl\{\wt{\mathfrak{A}}^E_\kappa
(u^+_{\nu-\zeta_k,k}) \mid
\xi_k \neq 0\bigr\},
\]
and for a nonvanishing homology class $\alpha \in H_{2\nu+1}$, we
define the spectral number
\[
\sigma(\alpha):=\inf \{\sigma(\xi) \mid \alpha=[\xi] \}.
\]
We next discuss under which conditions the spectral numbers are finite.

\begin{Proposition}\label{torprop2}
For $0\neq\alpha \in H_{2\nu+1}^B$, the spectral number $\sigma(\alpha)$
is finite if and only if
\[
E=\sum_{i=1}^m \bigl(\kappa_i-\tfrac{\hbar}{2}\bigr) \varpi_i.
\]
\end{Proposition}
\begin{proof}
We first introduce some useful notation. We define
\[
\iota \colon \ \mathbb{Z}+\tfrac{m-1}{2} \to \{1,\dots,m\}
\]
by the condition that
\[
\varepsilon(\nu)=e_{\iota(\nu)}.
\]
Note that in view of (\ref{ant3}), we have
\[
\iota(\nu)=\iota(-\nu), \qquad \nu \in \mathbb{Z}+\tfrac{m+1}{2}.
\]
Recall that an element in $B_{2\nu+1}$ has the form
\[
  \xi = \sum_{\ell \in \mathbb{Z}}\xi_\ell u^+_{\nu-\ell,\beta^+_\nu(\ell)}.
\]
Since $\wt{\mathfrak{A}}^E_\kappa\bigl(u^+_{\nu-\ell,\beta^+_\nu(\ell)}\bigr) =
  \mathfrak{A}^E_\kappa\bigl(u^+_{\nu-\ell,\beta^+_\nu(\ell)}\bigr) + \frac{(\nu-\ell)\hbar}{2}$,
its spectral number is given by
\begin{equation}\label{eq:sigma}
  \sigma(\xi) = \sup\bigl\{\Sigma(\ell)+\tfrac{\nu\hbar}{2}\mid \xi_\ell\neq 0\bigr\}
\end{equation}
with the function
\[
\Sigma \colon \ \mathbb{Z} \to \mathbb{R}, \qquad
\ell \mapsto \mathfrak{A}^E_\kappa\bigl(u^+_{\nu-\ell,\beta^+_\nu(\ell)}\bigr)-\tfrac{\ell\hbar}{2}.
\]
For $\ell \in \mathbb{N}$, the definition of $\beta_\nu^+$ yields
\[
  \beta^+_\nu(\ell) =  \sum_{j=1}^\ell\varepsilon(\nu+1-j) =
  \beta^+_{\nu-1}(\ell) + \varepsilon(\nu) - \varepsilon(\nu-\ell).
\]
Using this as well as~\eqref{wirk}, \eqref{ant3}, \eqref{ant1}
and~\eqref{varsigma}, we
compute for $\ell \in \mathbb{N}$,
\begin{align*}
\Sigma(\ell)
&= \langle \kappa,\beta^+_\nu(\ell)\rangle+
E\Omega_{\nu-\ell}-\tfrac{\ell\hbar}{2}\\ \nonumber
&= \bigl\langle \kappa,\beta^+_{\nu-1}(\ell) + \varepsilon(\nu) -
\varepsilon(\nu-\ell)\bigr\rangle + E\Omega_{\nu-\ell}-\tfrac{\ell\hbar}{2}\\ \nonumber
&= \bigl\langle \kappa,\beta^+_{\nu-1}(\ell) + \varepsilon(-\nu) -
\varepsilon(\ell-\nu)\bigr\rangle - E\Omega_{\ell-\nu}-\tfrac{\ell\hbar}{2}\\ \nonumber
&= \kappa_{\iota(-\nu)}-\kappa_{\iota(\ell-\nu)}+
\sum_{i=1}^m \bigl(\kappa_i-\tfrac{\hbar}{2}\bigr) \langle \beta^+_{\nu-1}(\ell),e_i
\rangle-E\Omega_{\ell-\nu}.
\end{align*}
Since we can identify the abelian groups
$H^B_{2\nu+1}$ for different $\nu$ using the
$\mathfrak{u}$-maps, which shift the action by some constant amount, it suffices
to prove the proposition for the special case
\[
\nu=-\tfrac{m-1}{2}.
\]
In this case, for $\ell \in \mathbb{N}$ the preceding computation gives
\begin{equation}\label{sigfun}
\Sigma(\ell)
= \kappa_{\iota (\frac{m-1}{2} )}-\kappa_{\iota (\frac{m-1}{2}+\ell )}+
\sum_{i=1}^m \bigl(\kappa_i-\tfrac{\hbar}{2}\bigr) \big\langle \beta^+_{-\frac{m+1}{2}}(\ell),e_i
\big\rangle-E\Omega_{\frac{m-1}{2}+\ell}.
\end{equation}
For $1\leq i\leq m$, let $\lambda_i$ be the unique positive
integer such that
\[
  \frac{\lambda_i}{\varpi_i} \leq \Omega_{\frac{m-1}{2}+\ell}
  < \frac{\lambda_i+1}{\varpi_i}.
\]
This means that $\lambda_i$ is the number of multiples of
$\tfrac{1}{\varpi_i}$ appearing in the nondecreasing sequence
$\Omega_{\frac{m-1}{2}+1},\dots,\Omega_{\frac{m-1}{2}+\ell}$, hence
\begin{align*}
  \lambda_i
  = \#\bigl\{j\in\{1,\dots,\ell\}\mid \varepsilon\bigl(\tfrac{m-1}{2}+j\bigr) = e_i\bigr\}
  =\sum_{j=1}^\ell \bigl\langle\varepsilon\bigl(\tfrac{m-1}{2}+j\bigr),e_i\bigr\rangle
   = \bigl\langle\beta^+_{-\frac{m+1}{2}}(\ell),e_i\bigr\rangle.
\end{align*}
This yields the inequalities
\begin{equation}\label{mainest}
\bigl\langle\beta^+_{-\frac{m+1}{2}}(\ell),e_i
\bigr\rangle\leq \varpi_i \cdot \Omega_{\frac{m-1}{2}+\ell}
< \bigl\langle\beta^+_{-\frac{m+1}{2}}(\ell),e_i
\bigr\rangle+1, \qquad 1 \leq i \leq m,
\end{equation}
and for $i=\iota(\frac{m-1}{2}+\ell)$ the equality
\begin{equation*}
  \varpi_{\iota(\frac{m-1}{2}+\ell)}\cdot \Omega_{\frac{m-1}{2}+\ell}
  = \bigl\langle\beta^+_{-\frac{m+1}{2}}(\ell),e_{\iota(\frac{m-1}{2}+\ell)}\bigl\rangle.
\end{equation*}
We further introduce the function
\[
\Gamma \colon \ \mathbb{N} \to \mathbb{R}, \qquad
\ell \mapsto \Sigma(\ell)+\bigg(E-\sum_{i=1}^m \bigl(\kappa_i-\tfrac{\hbar}{2}\bigr)\varpi_i\bigg)
\Omega_{\frac{m-1}{2}+\ell}\,.
\]
Using (\ref{mainest}) and (\ref{sigfun}), we estimate for $\ell \in
\mathbb{N}$ from above,
\begin{align*}
\Gamma(\ell)&=\Sigma(\ell)+E\Omega_{\frac{m-1}{2}+\ell}-\sum_{i=1}^m\bigl(\kappa_i
-\tfrac{\hbar}{2}\bigr)
\varpi_i \Omega_{\frac{m-1}{2}+\ell}\\
&\leq \Sigma(\ell)+E\Omega_{\frac{m-1}{2}+\ell}-\sum_{i=1}^m\bigl(\kappa_i-\tfrac{\hbar}{2}\bigr)
\bigl\langle \beta^+_{-\frac{m+1}{2}}(\ell),e_i\bigr\rangle
 =\kappa_{\iota(\frac{m-1}{2})}-\kappa_{\iota(\frac{m-1}{2}+\ell)},
\end{align*}
and similarly from below,
\begin{align*}
\Gamma(\ell)
&\geq \Sigma(\ell)+E\Omega_{\frac{m-1}{2}+\ell}-\sum_{i=1}^m\bigl(\kappa_i-\tfrac{\hbar}{2}\bigr)
\bigl\langle \beta^+_{-\frac{m+1}{2}}(\ell),e_i\bigr\rangle -\sum_{i=1}^m \bigl(\kappa_i
-\tfrac{\hbar}{2}\bigr)\\ \nonumber
&= \kappa_{\iota (\frac{m-1}{2} )}-\kappa_{\iota (\frac{m-1}{2}+\ell )}
-\sum_{i=1}^m \bigl(\kappa_i-\tfrac{\hbar}{2}\bigr).
\end{align*}
So for $\ell\in\N$, we have
\begin{equation}\label{eq:Sigma}
  \Sigma(\ell)=\Gamma(\ell)+\biggl(\sum_{i=1}^m
  \bigl(\kappa_i-\tfrac{\hbar}{2}\bigr)\varpi_i-E\biggr)\Omega_{\frac{m-1}{2}+\ell},
\end{equation}
where the function $\Gamma(\ell)$ is uniformly bounded from above and
below,
and $\Omega_{\frac{m-1}{2}+\ell}$ is nondecreasing in $\ell$ with
\[
\lim_{\ell \to \infty}\Omega_{\frac{m-1}{2}+\ell}=\infty.
\]
Now we argue as in the proof of Proposition~\ref{spec}, considering
$\xi\in B_{2\nu+1}$ with $[\xi]=\alpha\neq 0$ and distinguishing
three cases.

\textit{Case $E<\sum_{i=1}^m\bigl(\kappa_i-\tfrac{\hbar}{2}\bigr)\varpi_i$:}
Since $\xi$ has infinite positive support by Proposition~\ref{torprop1},
we conclude from~\eqref{eq:sigma} and~\eqref{eq:Sigma} that
$\sigma(\xi)=\infty$ and thus $\sigma(\alpha)=\infty$.

\textit{Case $E=\sum_{i=1}^m\bigl(\kappa_i-\tfrac{\hbar}{2}\bigr)\varpi_i$:}
As in the proof of Proposition~\ref{spec}, we have
$\bigl[\xi^L\bigr]=[\xi]=\alpha$ for $\xi^L:=\sum_{\ell \geq L}\xi_\ell u^+_{\nu-\ell,\beta^+_\nu(\ell)}$
and it follows that $\sigma(\alpha)$ is the finite quantity
\[
  \sigma(\alpha) = \lim_{L\to\infty}\sup\bigl\{\Gamma(\ell)+\tfrac{\nu\hbar}{2}\mid \ell\geq
  L,\, \xi_\ell\neq 0\bigr\}.
\]

\textit{Case $E>\sum_{i=1}^m\bigl(\kappa_i-\tfrac{\hbar}{2}\bigr)\varpi_i$:}
As in the proof of Proposition~\ref{spec}, it follows that
\[
  \sigma(\alpha) = \lim_{L\to\infty}\sup\bigl\{\Sigma(\ell)+\tfrac{\nu\hbar}{2}\mid \ell\geq
  L,\,\xi_\ell\neq 0\bigr\} = -\infty.
\]
This proves Proposition~\ref{torprop2}.
\end{proof}

We are now ready to prove Theorem~\ref{TheoremB} from the introduction.

\begin{proof}[Proof of Theorem~\ref{TheoremB}]
For $0\neq\alpha \in H_*$, we have
\[
\sigma(\mathfrak{u}_i\alpha)=\sigma(\alpha)+\kappa_i, \qquad 1\leq i \leq m.
\]
Theorem~\ref{TheoremB} now follows from
Propositions~\ref{torprop1} and~\ref{torprop2}
by the same argument as in the proof of Theorem~\ref{TheoremA}.
\end{proof}

\pdfbookmark[1]{References}{ref}
\LastPageEnding

\end{document}